\documentclass[reqno,12pt]{amsart}
\usepackage{amsfonts}
\usepackage{bbm}
\usepackage{} 
\numberwithin{equation}{section}
\usepackage{indentfirst}
 \usepackage{color}
\usepackage{amssymb}
\usepackage{mathrsfs}
\usepackage{xy}
\xyoption{all}
\def\Ext{\mbox{\rm Ext}\,} \def\Hom{\mbox{\rm Hom}} \def\dim{\mbox{\rm dim}\,} \def\Iso{\mbox{\rm Iso}\,}
\def\LHS{\mbox{\rm LHS}}\def\RHS{\mbox{\rm RHS}}
\def\lr#1{\langle #1\rangle}    
     
\def\tw{\mbox{\rm tw}\,}\def\ce{\mbox{\rm ce}\,}

\def\K{\mathcal {K}}\def\sK{\mathscr {K}}
\def\Aut{\mbox{\rm Aut}\,}\def\A{\mathcal{A}\,} 

\theoremstyle{plain} 
\newtheorem{theorem}{\bf Theorem}[section]
\newtheorem{lemma}[theorem]{\bf Lemma}
\newtheorem{corollary}[theorem]{\bf Corollary}
\newtheorem{proposition}[theorem]{\bf Proposition}

\theoremstyle{definition} 
\newtheorem{definition}[theorem]{\bf Definition}
\newtheorem{remark}[theorem]{\bf Remark}
\newtheorem{example}[theorem]{\bf Example}

\newcommand{\bt}{\begin{theorem}}
\newcommand{\et}{\end{theorem}}
\newcommand{\bl}{\begin{lemma}}
\newcommand{\el}{\end{lemma}}
\newcommand{\bd}{\begin{definition}}
\newcommand{\ed}{\end{definition}}
\newcommand{\bc}{\begin{corollary}}
\newcommand{\ec}{\end{corollary}}
\newcommand{\bp}{\begin{proof}}
\newcommand{\ep}{\end{proof}}
\newcommand{\bx}{\begin{example}}
\newcommand{\ex}{\end{example}}
\newcommand{\br}{\begin{remark}}
\newcommand{\er}{\end{remark}}
\newcommand{\be}{\begin{equation}}
\newcommand{\ee}{\end{equation}}
\newcommand{\ba}{\begin{align}}
\newcommand{\ea}{\end{align}}
\newcommand{\bn}{\begin{enumerate}}
\newcommand{\en}{\end{enumerate}}
\newcommand{\bcs}{\begin{cases}}
\newcommand{\ecs}{\end{cases}}
\makeatletter
\renewcommand{\section}{\@startsection{section}{1}{0mm}
  {-\baselineskip}{0.5\baselineskip}{\bf\leftline}}
\makeatother

\begin{document}

\title[Drinfeld doubles, derived Hall algebras, Bridgeland Hall algebras]{Drinfeld doubles via derived Hall algebras \\
and Bridgeland Hall algebras} 

\author{Fan Xu and Haicheng Zhang} 


\subjclass[2010]{ 
16G20, 17B20, 17B37.
}
%
\keywords{ 
Heisenberg double; Drinfeld double; derived Hall algebra; Bridgeland Hall algebra.
}
\address{
Department of Mathematical Sciences, Tsinghua University, Beijing 100084, P. R. China.\endgraf}
\email{fanxu@mail.tsinghua.edu.cn}

\address{
Institute of Mathematics, School of Mathematical Sciences, Nanjing Normal University,
 Nanjing 210023, P. R. China.\endgraf}
\email{zhanghc@njnu.edu.cn}


\maketitle

\begin{abstract}
Let $\A$ be a finitary hereditary abelian category. We give a Hall algebra presentation of Kashaev's theorem on the relation between Drinfeld double and Heisenberg double. As applications, we obtain realizations of the Drinfeld double Hall algebra of $\A$ via its derived Hall algebra and Bridgeland Hall algebra of $m$-cyclic complexes.
\end{abstract}

\section{Introduction}
The Hall algebra $\mathfrak{H}(A)$ of a finite dimensional algebra $A$ over a finite field was introduced by Ringel \cite{R90a} in 1990. Ringel \cite{R90,R90a} proved that if $A$ is a hereditary algebra of finite type, the twisted Hall algebra $\mathfrak{H}_{v}(A)$, called the Ringel--Hall algebra, is isomorphic to the positive part of the corresponding quantized enveloping algebra. In 1995,
Green \cite{Gr95} generalized Ringel's work to any hereditary algebra $A$
and showed that the composition subalgebra of $\mathfrak{H}_{v}(A)$ generated by simple $A$-modules
gives a realization of the positive part of the quantized enveloping algebra associated with $A$. Moreover, he introduced a bialgebra structure on $\mathfrak{H}_{v}(A)$ via a significant formula called Green's formula. In 1997, Xiao \cite{Xiao} provided the antipode on $\mathfrak{H}_{v}(A)$, and proved that the extended Ringel--Hall algebra is a Hopf algebra. Furthermore, he considered the Drinfeld double of the extended Ringel--Hall algebras, and obtained a realization of the full quantized enveloping algebra.

In order to give an intrinsic realization of the entire quantized enveloping algebra via Hall algebra approach, one tried to define the Hall algebra of a triangulated category (for example, \cite{Kapranov}, \cite{Toen}, \cite{XiaoXu}).
Kapranov \cite{Kapranov} considered the Heisenberg double of the extended Ringel--Hall algebras, and defined an associative algebra, called the lattice algebra, for the bounded derived category of a hereditary algebra $A$. By using the fibre products of model categories, To\"{e}n \cite{Toen} defined an associative algebra, called the derived Hall algebra, for a DG-enhanced triangulated category.
Later on, Xiao and Xu \cite{XiaoXu} generalized the definition of the derived Hall algebra to any triangulated category with some homological finiteness conditions. In particular, the derived Hall algebra $\mathcal {D}\mathcal {H}(A)$ of the bounded derived category of a hereditary algebra $A$ can be defined, and it is proved in \cite{Toen} that there exist certain Heisenberg double structures in $\mathcal {D}\mathcal {H}(A)$.

Recently, for each hereditary algebra $A$, Bridgeland \cite{Br} defined an associative algebra, called the Bridgeland Hall algebra, which is the Ringel--Hall algebra of $2$-cyclic complexes over projective $A$-modules with some localization and reduction. He proved that the quantized enveloping algebra associated to \emph{A} can be embedded into its Bridgeland Hall algebra. This provides a beautiful realization of the full quantized enveloping algebra by Hall algebras. Afterwards, Yanagida \cite{Yan} (see also \cite{ZHC}) showed that the Bridgeland Hall algebra of $2$-cyclic complexes of a hereditary algebra is isomorphic to the Drinfeld double of its extended Ringel--Hall algebras. Inspired by the work of Bridgeland, Chen and Deng \cite{ChenD} introduced the Bridgeland Hall algebra $\mathcal {D}\mathcal {H}_m(A)$ of $m$-cyclic complexes of a hereditary algebra $A$ for each nonnegative integer $m\neq1$. If $m=0$ or $m>2$, the algebra structure of $\mathcal {D}\mathcal {H}_m(A)$ has a characterization in \cite{ZHC2}, in particular, it is proved that there exist Heisenberg double structures in $\mathcal {D}\mathcal {H}_m(A)$.

Kashaev \cite{Kas} established a relation between the Drinfeld double and Heisenberg double of a Hopf algebra. Explicitly, he showed that the Drinfeld double is representable as a subalgebra in the tensor square of the Heisenberg double.

In this paper, let $\A$ be a finitary hereditary abelian category. We first give a Hall algebra presentation of Kashaev's Theorem on the relation between Drinfeld double and Heisenberg double. Then we apply this presentation to the Bridgeland Hall algebra and derived Hall algebra of $\A$.

Throughout the paper, all tensor products are assumed to be  over the complex number field $\mathbb{C}$. Let $k$ be a fixed finite field with $q$ elements and set $v=\sqrt{q}\in \mathbb{C}$. Let $\A$ be a finitary hereditary abelian $k$-category. We denote by $\Iso(\A)$ and $K(\A)$ the set of isoclasses of objects in $\A$ and the Grothendieck group of $\A$, respectively. For each object $M$ in $\A$, the class of $M$ in $K(\A)$ is denoted by $\hat{M}$, and the automorphism group of $M$ is denoted by $\Aut(M)$. For a finite set $S$, we denote by $|S|$ its cardinality, and we also write $a_M$ for $|\Aut(M)|$. For a positive integer $m$, we denote the quotient ring $\mathbb{Z}/m\mathbb{Z}$ by $\mathbb{Z}_m=\{0,1,\ldots,m-1\}$. By convention, $\mathbb{Z}_0=\mathbb{Z}$.

\section{Preliminaries}
In this section, we recall the definitions of Ringel--Hall algebra, Heisenberg double, and Drinfeld double (cf. \cite{Sch,Xiao,Kapranov}).
\subsection{Hall algebras}
For objects $M,N_1,\ldots,N_t\in\A$, let $g_{N_1\cdots N_t}^M$ be the number of the filtrations
$$M=M_0\supseteq M_1\supseteq\cdots\supseteq M_{t-1}\supseteq M_t=0$$
such that $M_{i-1}/M_i\cong N_i$ for all $1\leq i\leq t$. In particular, if $t=2$, $g_{N_1N_2}^M$ is the number of subobjects $X$ of $M$ such that $X\cong N_2$ and $M/X\cong N_1$.
One defines the \emph{Hall algebra} $\mathfrak{H}(\A)$ to be the vector space over $\mathbb{C}$ with basis $[M]\in \Iso(\A)$ and with the multiplication defined by
\[[M] \diamond [N] = \sum\limits_{[L]} g_{MN}^L[L].\]
By definition, it is easy to see that for each $1< i< t$,
$$g_{N_1\cdots N_t}^M=\sum\limits_{[X]}g_{N_1\cdots N_{i-1}X}^Mg_{N_{i}\cdots N_t}^X=\sum\limits_{[Y]}g_{N_1\cdots N_{i}}^Yg_{YN_{i+1}\cdots N_t}^M.$$

For any $M, N\in\A$, define $$\lr{M,N}:=\dim_k{\Hom_{\A}(M,N)}-\dim_k{\Ext_{\A}^{1}(M,N)}.$$
It induces a bilinear form
$$\lr{\cdot ,\cdot }: K(\A)\times K(\A)\longrightarrow \mathbb{Z},$$ known as the \emph{Euler form}. We also consider the \emph{symmetric Euler form}
$$(\cdot ,\cdot ): K(\A)\times K(\A)\longrightarrow \mathbb{Z},$$ defined by $(\alpha,\beta)=\lr{\alpha,\beta}+\lr{\beta,\alpha}$ for all $\alpha,\beta \in K(\A)$.

The twisted Hall algebra $\mathfrak{H}_v(\A)$, called the \emph{Ringel--Hall algebra}, is the same vector space as $\mathfrak{H}(\A)$ but with the twisted multiplication defined by $$[M][N]=v^{\lr{{M},\,{N}}}\cdot[M]\diamond[N].$$
We can form the \emph{extended Ringel--Hall algebra} $\mathfrak{H}_v^e(\A)$ by adjoining symbols $K_\alpha$ for all $\alpha\in K(\A)$ and imposing relations
\begin{equation}
K_\alpha K_\beta=K_{\alpha+\beta},~~~~K_\alpha[M]=v^{(\alpha,\,\hat{M})}\cdot [M]K_\alpha.
\end{equation}

Green \cite{Gr95} introduced a (topological) bialgebra structure on $\mathfrak{H}_v^e(\A)$ by defining the comultiplication as follows:
$$\Delta([L]K_{\alpha})
=\sum\limits_{[M],[N]}v^{\lr{M,N}}\frac{a_{M}a_{N}}{a_{L}}g_{MN}^{L}[M]K_{\hat{N}+\alpha}\otimes [N]K_{\alpha}, \text{\ for\ any\ }  L\in\A, \alpha\in K(\A).$$
That $\Delta$ is a homomorphism of algebras amounts to the following crucial formula.
\begin{theorem}{\rm\textbf{(Green's formula)}}
Given $M,N,M',N'\in\A$, we have the following formula
\begin{equation}
\begin{split}
&a_Ma_Na_{M'}a_{N'}\sum\limits_{[L]}g_{MN}^Lg_{M'N'}^L\frac{1}{a_L}\\&=
\sum\limits_{[A],[A'],[B],[B']}\frac{|\Ext_{\A}^1(A,B')|}{|\Hom_{\A}(A,B')|}
g_{AA'}^{M}g_{BB'}^{N}g_{AB}^{M'}g_{A'B'}^{N'}a_{A}a_{A'}a_{B}a_{B'}.
\end{split}
\end{equation}
\end{theorem}
\subsection{Heisenberg doubles}
Let $A$ and $B$ be Hopf algebras, and let $\varphi:A\times B\to\mathbb{C}$ be a Hopf pairing. The \emph{Heisenberg double} $HD(A,B,\varphi)$ is defined to be the free product $A\ast B$ imposed by the following relations (with $a\in A$ and $b\in B$):
$$b\ast a=\sum\varphi(a_2,b_1)a_1\ast b_2,$$
where and elsewhere we use Sweedler's notation $\Delta(a)=\sum a_1\otimes a_2$.

There exists a so-called Green's pairing $\varphi_0:\mathfrak{H}_v^e(\A)\times\mathfrak{H}_v^e(\A)\to\mathbb{C}$ defined by
$$\varphi_0([M]K_\alpha,[N]K_\beta)=\delta_{[M],[N]}\frac{v^{(\alpha,\beta)}}{a_M},$$ which is a Hopf pairing.

Now let us apply the construction of Heisenberg double to Ringel--Hall algebras. Let $H^+(\A)$ (resp. $H^-(\A)$) be the Ringel--Hall algebra $\mathfrak{H}_v^e(\A)$ with each $[M]K_\alpha$ rewritten as $\mu_M^+K_\alpha^+$ (resp. $\mu_M^-K_\alpha^-$). Thus, considering $A=H^-(\A)$, $B=H^+(\A)$ and $\varphi=\varphi_0$, we obtain the \emph{Heisenberg double Hall algebra}, denoted by $HD(\A)$. By direct calculations, we give the characterization of $HD(\A)$ via generators and generating relations (with $\alpha,\beta\in K(\A)$ and $[M],[N]\in\Iso(\A)$) as follows (cf. \cite{Kapranov}):
\begin{flalign}
&\mu_M^+\mu_N^+=\sum\limits_{[L]}v^{\lr{M,N}}g_{MN}^L\mu_L^+,\quad\quad\mu_M^-\mu_N^-=\sum\limits_{[L]}v^{\lr{M,N}}g_{MN}^L\mu_L^-;\label{a1}\\
&K_{\alpha}^+\mu_M^+=v^{(\alpha,\hat{M})}\mu_M^+K_{\alpha}^+,\quad\quad K_{\alpha}^-\mu_M^-=v^{(\alpha,\hat{M})}\mu_M^-K_{\alpha}^-;\label{a2}\\
&K_\alpha^\pm K_\beta^\pm=K_{\alpha+\beta}^\pm,\quad\quad K_\alpha^+ K_\beta^-=v^{(\alpha,\beta)}K_\beta^-K_\alpha^+;\label{a3}\\
&K_\alpha^+\mu_M^-=\mu_M^-K_\alpha^+,\quad\quad K_\alpha^-\mu_M^+=v^{-(\alpha,\hat{M})}\mu_M^+K_\alpha^-;\label{a4}\\
&\mu_M^+\mu_N^-=\sum\limits_{[X],[Y]}v^{\lr{\hat{N}-\hat{Y},\hat{X}-\hat{Y}}}\gamma_{MN}^{XY}K_{\hat{N}-\hat{Y}}^-\mu_Y^-\mu_X^+;\label{a5}
\end{flalign}
where and elsewhere $\gamma_{MN}^{XY}=\frac{a_Xa_Y}{a_Ma_N}\sum\limits_{[L]}a_Lg_{LX}^Mg_{YL}^N$.

Similarly, one defines the \emph{dual Heisenberg double Hall algebra} $\check{H}D(\A)$, which is given by the generators and generating relations
(with $\alpha,\beta\in K(\A)$ and $[M],[N]\in\Iso(\A)$) as follows:
\begin{flalign}
&\nu_M^+\nu_N^+=\sum\limits_{[L]}v^{\lr{M,N}}g_{MN}^L\nu_L^+,\quad\quad\nu_M^-\nu_N^-=\sum\limits_{[L]}v^{\lr{M,N}}g_{MN}^L\nu_L^-;\\
&\K_{\alpha}^+\nu_M^+=v^{(\alpha,\hat{M})}\nu_M^+\K_{\alpha}^+,\quad\quad \K_{\alpha}^-\nu_M^-=v^{(\alpha,\hat{M})}\nu_M^-\K_{\alpha}^-;\\
&\K_\alpha^\pm \K_\beta^\pm=\K_{\alpha+\beta}^\pm,\quad\quad \K_\alpha^+ \K_\beta^-=v^{-(\alpha,\beta)}\K_\beta^-\K_\alpha^+;\\
&\K_\alpha^-\nu_M^+=\nu_M^+\K_\alpha^-,\quad\quad \K_\alpha^+\nu_M^-=v^{-(\alpha,\hat{M})}\nu_M^-\K_\alpha^+;\\
&\nu_N^-\nu_M^+=\sum\limits_{[X],[Y]}v^{\lr{\hat{N}-\hat{Y},\hat{Y}-\hat{X}}}\gamma_{NM}^{YX}\K_{\hat{N}-\hat{Y}}^+\nu_X^+\nu_Y^-\label{b5}.
\end{flalign}

\subsection{Drinfeld doubles}
Let $A$ and $B$ be Hopf algebras, and let $\varphi:A\times B\to\mathbb{C}$ be a Hopf pairing. The \emph{Drinfeld double} $D(A,B,\varphi)$ is defined to be the free product $A*B$ imposed by the following relations (with $a\in A$ and $b\in B$):
\begin{equation}\sum\varphi(a_1,b_2)b_1\ast a_2=\sum\varphi(a_2,b_1)a_1\ast b_2.\end{equation}

Applying the construction of Drinfeld double to the Ringel--Hall algebras $H^-(\A)$ and $H^+(\A)$, we obtain the \emph{Drinfeld double Hall algebra}, denoted by $D(\A)$, which is defined by the generators and generating relations (with $\alpha,\beta\in K(\A)$, $[M],[N]\in\Iso(\A)$) as follows:
\begin{flalign}
&\omega_M^+\omega_N^+=\sum\limits_{[L]}v^{\lr{M,N}}g_{MN}^L\omega_L^+,\quad\quad\omega_M^-\omega_N^-=\sum\limits_{[L]}v^{\lr{M,N}}g_{MN}^L\omega_L^-\label{d1};\\
&\sK_{\alpha}^+\omega_M^+=v^{(\alpha,\hat{M})}\omega_M^+\sK_{\alpha}^+,\quad\quad \sK_{\alpha}^-\omega_M^-=v^{(\alpha,\hat{M})}\omega_M^-\sK_{\alpha}^-\label{d2};\\
&\sK_\alpha^\pm \sK_\beta^\pm=\sK_{\alpha+\beta}^\pm,\quad\quad \sK_\alpha^+\sK_\beta^-=\sK_\beta^-\sK_\alpha^+\label{d3};\\
&\sK_\alpha^+\omega_M^-=v^{-(\alpha,\hat{M})}\omega_M^-\sK_\alpha^+,\quad\quad \sK_\alpha^-\omega_M^+=v^{-(\alpha,\hat{M})}\omega_M^+\sK_\alpha^-\label{d4};\\
&\sum\limits_{[X],[Y]}v^{\lr{\hat{M}-\hat{X},\hat{M}-\hat{N}}}\gamma_{MN}^{XY}\sK_{\hat{M}-\hat{X}}^-\omega_Y^-\omega_X^+=
\sum\limits_{[X],[Y]}v^{\lr{\hat{M}-\hat{X},\hat{N}-\hat{M}}}\gamma_{NM}^{YX}\sK_{\hat{M}-\hat{X}}^+\omega_X^+\omega_Y^-\label{d5}.
\end{flalign}

\section{Kashaev's theorem: Hall algebra presentation}
In this section, we prove Kashaev's theorem \cite[Theorem 2]{Kas} in the form of Ringel--Hall algebras. There are some similar constructions in \cite{Kap}, but they are not so natural.
\begin{theorem}\label{zyjg}
There exists an embedding of algebras $I:D(\A)\hookrightarrow HD(\A)\otimes\check{H}D(\A)$ defined on generators by
\begin{flalign*}
\sK_\alpha^+\mapsto K_\alpha^+\otimes\K_\alpha^+,\quad \omega_M^+\mapsto\sum\limits_{[M_1],[M_2]}v^{\lr{M_1,M_2}}
\frac{a_{M_1}a_{M_2}}{a_M}g_{M_1M_2}^M\mu_{M_1}^+K_{\hat{M}_2}^+\otimes\nu_{M_2}^+,
\end{flalign*} and
\begin{flalign*}
\sK_\alpha^-\mapsto K_\alpha^-\otimes\K_\alpha^-,\quad \omega_M^-\mapsto\sum\limits_{[M_1],[M_2]}v^{\lr{M_2,M_1}}
\frac{a_{M_1}a_{M_2}}{a_M}g_{M_2M_1}^M\mu_{M_1}^-\otimes\nu_{M_2}^-\K_{\hat{M}_1}^-.
\end{flalign*}
\end{theorem}
\begin{proof}
In order to prove that $I$ is a homomorphism of algebras, it suffices to show that the relations from $(\ref{d1})$ to $(\ref{d5})$ are preserved under $I$. We only prove the relations $(\ref{d1})$ and $(\ref{d5})$, since the other relations can be easily proved.

For the first relation in $(\ref{d1})$,
\begin{flalign*}
&\sum\limits_{[L]}v^{\lr{M,N}}g_{MN}^LI(\omega_L^+)=
\sum\limits_{[L],[L_1],[L_2]}v^{\lr{M,N}+\lr{L_1,L_2}}\frac{a_{L_1}a_{L_2}}{a_L}g_{MN}^Lg_{L_1L_2}^L\mu_{L_1}^+K_{\hat{L}_2}^+\otimes\nu_{L_2}^+.
\end{flalign*}
\begin{flalign*}
&I(\omega_M^+)I(\omega_N^+)=\\
&\sum\limits_{[M_1],[M_2],[N_1],[N_2]}v^{\lr{M_1,M_2}+\lr{N_1,N_2}}\frac{a_{M_1}a_{M_2}a_{N_1}a_{N_2}}{a_Ma_N}g_{M_1M_2}^Mg_{N_1N_2}^N
\mu_{M_1}^+K_{\hat{M}_2}^+\mu_{N_1}^+K_{\hat{N}_2}^+\otimes\nu_{M_2}^+\nu_{N_2}^+
\end{flalign*}
\begin{flalign*}
=\sum\limits_{[M_1],[M_2],[N_1],[N_2]}
v^{x_0}\frac{a_{M_1}a_{M_2}a_{N_1}a_{N_2}}{a_Ma_N}&g_{M_1M_2}^Mg_{N_1N_2}^N
\mu_{M_1}^+\mu_{N_1}^+K_{\hat{M}_2+\hat{N}_2}^+\otimes\nu_{M_2}^+\nu_{N_2}^+\\&(x_0=\lr{M_1,M_2}+\lr{N_1,N_2}+(M_2,N_1))
\end{flalign*}
\begin{flalign*}
=\sum\limits_{[M_1],[M_2],[N_1],[N_2],[L_1],[L_2]}
&v^{x_1}\frac{a_{M_1}a_{M_2}a_{N_1}a_{N_2}}{a_Ma_N}g_{M_1M_2}^Mg_{N_1N_2}^N
g_{M_1N_1}^{L_1}g_{M_2N_2}^{L_2}
\mu_{L_1}^+K_{\hat{L}_2}^+\otimes\nu_{L_2}^+\quad\quad(*)\\
&(x_1=\lr{M_1,M_2}+\lr{N_1,N_2}+(M_2,N_1)+\lr{M_1,N_1}+\lr{M_2,N_2}).
\end{flalign*}
For each fixed $L_1,L_2$, noting that in $(*)$ $\hat{M}=\hat{M}_1+\hat{M}_2$, $\hat{N}=\hat{N}_1+\hat{N}_2$, $\hat{L}_i=\hat{M}_i+\hat{N}_i$ for $i=1,2$, we obtain that $x_1=\lr{M,N}+\lr{L_1,L_2}-2\lr{M_1,N_2}$. Thus, by Green's formula, we conclude that
\begin{flalign*}
\sum\limits_{[M_i],[N_i],i=1,2}
v^{x_1}\frac{a_{M_1}a_{M_2}a_{N_1}a_{N_2}}{a_Ma_N}g_{M_1M_2}^Mg_{N_1N_2}^N
g_{M_1N_1}^{L_1}g_{M_2N_2}^{L_2}=\sum\limits_{[L]}v^{\lr{M,N}+\lr{L_1,L_2}}\frac{a_{L_1}a_{L_2}}{a_L}g_{MN}^Lg_{L_1L_2}^L
\end{flalign*}
and thus $$I(\omega_M^+)I(\omega_N^+)=\sum\limits_{[L]}v^{\lr{M,N}}g_{MN}^LI(\omega_L^+).$$
Similarly, we can prove that the second relation in (\ref{d1}) is also preserved under $I$.

Now, we come to prove that the relation in (\ref{d5}) is preserved under $I$. First of all, substituting $\gamma_{MN}^{XY}=\frac{a_Xa_Y}{a_Ma_N}\sum\limits_{[L]}a_Lg_{LX}^Mg_{YL}^N$ into (\ref{d5}), we rewrite (\ref{d5}) as follows:
\begin{flalign*}
\sum\limits_{[X],[Y],[L]}v^{\lr{\hat{L},\hat{M}-\hat{N}}}a_Xa_Ya_Lg_{LX}^Mg_{YL}^N\sK_{\hat{L}}^-\omega_Y^-\omega_X^+
=\sum\limits_{[X],[Y],[L]}v^{\lr{\hat{L},\hat{N}-\hat{M}}}a_{X}a_{Y}a_{L}g_{XL}^Mg_{LY}^N\sK_{\hat{L}}^+\omega_{X}^+\omega_{Y}^-.\end{flalign*}
On the one hand,
\begin{flalign*}
&\LHS:=\sum\limits_{[X],[Y],[L]}v^{\lr{\hat{L},\hat{M}-\hat{N}}}a_Xa_Ya_Lg_{LX}^Mg_{YL}^NI(\sK_{\hat{L}}^-)I(\omega_Y^-)I(\omega_X^+)\\
&=\sum\limits_{{\tiny\begin{array}{cc}
&[X],[Y],[L],\\ [-0.6ex]
&[Y_1],[Y_2],[X_1],[X_2]\end{array}}}v^{y_0}a_{X_1}a_{X_2}a_{Y_1}a_{Y_2}a_Lg_{LX}^Mg_{X_1X_2}^Xg_{Y_2Y_1}^Yg_{YL}^N
K_{\hat{L}}^-\mu_{Y_1}^-\mu_{X_1}^+K_{\hat{X}_2}^+\otimes\K_{\hat{L}}^-\nu_{Y_2}^-\K_{\hat{Y}_1}^-\nu_{X_2}^+\\
&\quad\quad\quad\quad\quad\quad\quad\quad\quad\quad\quad\quad\quad(y_0=\lr{\hat{L},\hat{M}-\hat{N}}+\lr{X_1,X_2}+\lr{Y_2,Y_1})\\
&=\sum\limits_{[L],[X_1],[X_2],[Y_1],[Y_2]}v^{y_1}a_{X_1}a_{X_2}a_{Y_1}a_{Y_2}a_Lg_{LX_1X_2}^Mg_{Y_2Y_1L}^N
K_{\hat{L}}^-\mu_{Y_1}^-\mu_{X_1}^+K_{\hat{X}_2}^+\otimes\K_{\hat{Y}_1+\hat{L}}^-\nu_{Y_2}^-\nu_{X_2}^+\\
&\quad\quad\quad\quad\quad\quad\quad(y_1=y_0-(Y_1,Y_2)=\lr{\hat{L},\hat{M}-\hat{N}}+\lr{X_1,X_2}-\lr{Y_1,Y_2}).
\end{flalign*}
By (\ref{b5}), \begin{flalign*}\nu_{Y_2}^-\nu_{X_2}^+&=\sum\limits_{[A],[B]}v^{\lr{\hat{Y}_2-\hat{B},\hat{B}-\hat{A}}}\gamma_{Y_2X_2}^{BA}\K_{\hat{Y}_2-\hat{B}}^+\nu_A^+\nu_B^-
\\&=\sum\limits_{[A],[B],[C]}v^{\lr{\hat{C},\hat{B}-\hat{A}}}\frac{a_Aa_Ba_C}{a_{X_2}a_{Y_2}}g_{CB}^{Y_2}g_{AC}^{X_2}\K_{\hat{C}}^+\nu_A^+\nu_B^-.\end{flalign*}
\begin{flalign*}
&\text{Thus},~~\LHS=\\&\sum\limits_{[L],[X_1],[Y_1],[A],[B],[C]}v^{y_2}a_La_{X_1}a_{A}a_{C}a_{B}a_{Y_1}g_{LX_1AC}^Mg_{CBY_1L}^N
K_{\hat{L}}^-\mu_{Y_1}^-\mu_{X_1}^+K_{\hat{A}+\hat{C}}^+\otimes\K_{\hat{Y}_1+\hat{L}}^-\K_{\hat{C}}^+\nu_{A}^+\nu_{B}^-\\
&\quad(y_2=y_1+\lr{\hat{C},\hat{B}-\hat{A}}=\lr{\hat{L},\hat{M}-\hat{N}}+\lr{\hat{X}_1,\hat{A}+\hat{C}}-\lr{\hat{Y}_1,\hat{B}+\hat{C}}+
\lr{\hat{C},\hat{B}-\hat{A}})\\
&=\sum\limits_{[L],[X_1],[Y_1],[A],[B],[C]}v^{y_3}a_La_{X_1}a_{A}a_{C}a_{B}a_{Y_1}g_{LX_1AC}^Mg_{CBY_1L}^N
K_{\hat{L}}^-\mu_{Y_1}^-\mu_{X_1}^+K_{\hat{A}+\hat{C}}^+\otimes\K_{\hat{C}}^+\nu_{A}^+\nu_{B}^-\K_{\hat{Y}_1+\hat{L}}^-\\
&\quad(y_3=\lr{\hat{L},\hat{M}-\hat{N}}+\lr{\hat{X}_1,\hat{A}+\hat{C}}-\lr{\hat{Y}_1,\hat{B}+\hat{C}}+
\lr{\hat{C},\hat{B}-\hat{A}}+(\hat{L}+\hat{Y}_1,\hat{B}+\hat{C})).
\end{flalign*}
On the other hand,
\begin{flalign*}
&\RHS:=\sum\limits_{[X],[Y],[L]}v^{\lr{\hat{L},\hat{N}-\hat{M}}}a_{X}a_{Y}a_{L}g_{XL}^Mg_{LY}^NI(\sK_{\hat{L}}^+)I(\omega_{X}^+)I(\omega_{Y}^-)\\
&=\sum\limits_{{\tiny\begin{array}{cc}
&[X],[Y],[L],\\ [-0.6ex]
&[X_1],[X_2],[Y_1],[Y_2]\end{array}}}v^{z_0}a_{X_1}a_{X_2}a_{Y_1}a_{Y_2}a_Lg_{X_1X_2}^Xg_{XL}^Mg_{LY}^Ng_{Y_2Y_1}^Y
K_{\hat{L}}^+\mu_{X_1}^+K_{\hat{X}_2}^+\mu_{Y_1}^-\otimes\K_{\hat{L}}^+\nu_{X_2}^+\nu_{Y_2}^-\K_{\hat{Y}_1}^-\\
&\quad\quad\quad\quad\quad\quad\quad\quad\quad\quad(z_0=\lr{\hat{L},\hat{N}-\hat{M}}+\lr{X_1,X_2}+\lr{Y_2,Y_1})\\
&=\sum\limits_{[L],[X_1],[X_2],[Y_1],[Y_2]}v^{z_1}a_{X_1}a_{X_2}a_{Y_1}a_{Y_2}a_Lg_{X_1X_2L}^Mg_{LY_2Y_1}^N
K_{\hat{X}_2+\hat{L}}^+\mu_{X_1}^+\mu_{Y_1}^-\otimes\K_{\hat{L}}^+\nu_{X_2}^+\nu_{Y_2}^-\K_{\hat{Y}_1}^-\\
&\quad\quad(z_1=z_0-(X_1,X_2)=\lr{\hat{L},\hat{N}-\hat{M}}+\lr{Y_2,Y_1}-\lr{X_2,X_1}).
\end{flalign*}
By (\ref{a5}),
\begin{flalign*}
\mu_{X_1}^+\mu_{Y_1}^-&=\sum\limits_{[A],[B]}v^{\lr{\hat{Y}_1-\hat{B},\hat{A}-\hat{B}}}\gamma_{X_1Y_1}^{AB}K_{\hat{Y}_1-\hat{B}}^-\mu_B^-\mu_A^+\\
&=\sum\limits_{[A],[B],[C]}v^{\lr{\hat{C},\hat{A}-\hat{B}}}\frac{a_Aa_Ba_C}{a_{X_1}a_{Y_1}}g_{CA}^{X_1}g_{BC}^{Y_1}K_{\hat{C}}^-\mu_B^-\mu_A^+.
\end{flalign*}
\begin{flalign*}
&\text{Thus},~~\RHS=\\&\sum\limits_{[L],[X_2],[Y_2],[A],[B],[C]}v^{z_2}a_{C}a_{A}a_{X_2}a_La_{Y_2}a_{B}g_{CAX_2L}^Mg_{LY_2BC}^N
K_{\hat{X}_2+\hat{L}}^+K_{\hat{C}}^-\mu_{B}^-\mu_{A}^+\otimes\K_{\hat{L}}^+\nu_{X_2}^+\nu_{Y_2}^-\K_{\hat{B}+\hat{C}}^-\end{flalign*}
\begin{flalign*}
(z_2=z_1+\lr{\hat{C},\hat{A}-\hat{B}}=\lr{\hat{L},\hat{N}-\hat{M}}+\lr{\hat{Y}_2,\hat{B}+\hat{C}}-\lr{\hat{X}_2,\hat{A}+\hat{C}}+
\lr{\hat{C},\hat{A}-\hat{B}})\end{flalign*}
\begin{flalign*}
=\sum\limits_{[L],[X_2],[Y_2],[A],[B],[C]}v^{z_3}a_{C}a_{A}a_{X_2}a_La_{Y_2}a_{B}g_{CAX_2L}^Mg_{LY_2BC}^N
K_{\hat{C}}^-\mu_{B}^-\mu_{A}^+K_{\hat{X}_2+\hat{L}}^+\otimes\K_{\hat{L}}^+\nu_{X_2}^+\nu_{Y_2}^-\K_{\hat{B}+\hat{C}}^-\end{flalign*}
\begin{flalign*}
(z_3=\lr{\hat{L},\hat{N}-\hat{M}}+\lr{\hat{Y}_2,\hat{B}+\hat{C}}-\lr{\hat{X}_2,\hat{A}+\hat{C}}+
\lr{\hat{C},\hat{A}-\hat{B}}+(\hat{L}+\hat{X}_2,\hat{A}+\hat{C})).
\end{flalign*}
Identifying $L,X_1,A,C,B,Y_1$ in $\LHS$ with $C,A,X_2,L,Y_2,B$ in $\RHS$, respectively, we obtain that
$y_3=\lr{\hat{C},\hat{M}-\hat{N}}+\lr{\hat{A},\hat{X}_2+\hat{L}}-\lr{\hat{B},\hat{Y}_2+\hat{L}}+
\lr{\hat{L},\hat{Y}_2-\hat{X}_2}+(\hat{B}+\hat{C},\hat{Y}_2+\hat{L}).$
Noting that in $\RHS$ $\hat{M}-\hat{N}=\hat{X}-\hat{Y}=(\hat{X}_1-\hat{Y}_1)+(\hat{X}_2-\hat{Y}_2)=
(\hat{A}-\hat{B})+(\hat{X}_2-\hat{Y}_2)$, we have that
\begin{flalign*}y_3&=\lr{\hat{C},\hat{A}-\hat{B}}+\lr{\hat{C},\hat{X}_2}-\lr{\hat{C},\hat{Y}_2}+\lr{\hat{A},\hat{X}_2}+\lr{\hat{A},\hat{L}}
-\lr{\hat{B},\hat{Y}_2+\hat{L}}+\lr{\hat{L},\hat{Y}_2-\hat{X}_2}\\
&\quad+(\hat{C},\hat{L})+\lr{\hat{C},\hat{Y}_2}+\lr{\hat{Y}_2,\hat{C}}
+\lr{\hat{B},\hat{Y}_2+\hat{L}}+\lr{\hat{Y}_2,\hat{B}}+\lr{\hat{L},\hat{B}}\\
&=\lr{\hat{C},\hat{A}-\hat{B}}+\lr{\hat{A}+\hat{C},\hat{X}_2}+\lr{\hat{A},\hat{L}}+\lr{\hat{L},\hat{Y}_2-\hat{X}_2}+(\hat{C},\hat{L})
+\lr{\hat{Y}_2,\hat{B}+\hat{C}}+\lr{\hat{L},\hat{B}}
\end{flalign*}
and
\begin{flalign*}
z_3&=\lr{\hat{L},\hat{B}}-\lr{\hat{L},\hat{A}}+\lr{\hat{L},\hat{Y}_2-\hat{X}_2}+\lr{\hat{Y}_2,\hat{B}+\hat{C}}-
\lr{\hat{X}_2,\hat{A}+\hat{C}}+\lr{\hat{C},\hat{A}-\hat{B}}\\
&\quad+(\hat{C},\hat{L})+\lr{\hat{L},\hat{A}}+\lr{\hat{A},\hat{L}}+
\lr{\hat{X}_2,\hat{A}+\hat{C}}+\lr{\hat{A}+\hat{C},\hat{X}_2}\\
&=\lr{\hat{L},\hat{B}}+\lr{\hat{L},\hat{Y}_2-\hat{X}_2}+\lr{\hat{Y}_2,\hat{B}+\hat{C}}+\lr{\hat{C},\hat{A}-\hat{B}}
+(\hat{C},\hat{L})+\lr{\hat{A},\hat{L}}+\lr{\hat{A}+\hat{C},\hat{X}_2}\\
&=y_3.
\end{flalign*}
Hence, $\LHS=\RHS$, and we have proved that $I$ is a homomorphism of algebras.

Since $D(\A)\cong H^+(\A)\otimes H^-(\A)$ as a vector space, and the restriction of $I$ to the positive (negative) part is injective, we conclude that $I$ is injective. Therefore, we complete the proof.
\end{proof}

\section{Applications}
In this section, we apply Theorem \ref{zyjg} to Bridgeland Hall algebras of $m$-cyclic complexes and derived Hall algebras.

\subsection{Bridgeland Hall algebras}
Assume that $\A$ has enough projectives, the Bridgeland Hall algebra of $2$-cyclic complexes of $\A$ was introduced in \cite{Br}. Inspired by the work of Bridgeland, for each nonnegative integer $m\neq1$, Chen and Deng \cite{ChenD} introduced the Bridgeland Hall algebra $\mathcal {D}\mathcal {H}_m(\A)$ of $m$-cyclic complexes.
For $m=0$ or $m>2$, we recall the algebra structure of $\mathcal {D}\mathcal {H}_m(\A)$ by \cite{ZHC2} as follows:

\begin{proposition}\label{BR} {\rm(\cite{ZHC2})}
Let $m=0$ or $m>2$. Then
$\mathcal {D}\mathcal {H}_m(\A)$ is an associative and unital $\mathbb{C}$-algebra generated by the elements in $\{e_{M,i}~|~[M]\in\Iso(\A),~i\in \mathbb{Z}_m\}$ and $\{K_{\alpha,i}~|~\alpha\in K(\A),~i\in \mathbb{Z}_m\}$, and the following relations:
\begin{flalign}
&{K_{\alpha,i}} {K_{\beta,i}}={K_{\alpha+\beta,i}},~~~~
{K_{\alpha,i}} {K_{\beta,j}}=\begin{cases}
v^{( \alpha,\,\beta)}{K_{\beta,j}} {K_{\alpha,i}} \quad &\text{$i=j+1$},\\
v^{-( \alpha,\,\beta)}{K_{\beta,j}} {K_{\alpha,i}} \quad &\text{$i=m-1+j$},\\
{K_{\beta,j}} {K_{\alpha,i}} & {\text{otherwise};}
\end{cases}\\
&{K_{\alpha,i}}e_{M,j}=\begin{cases}
v^{(\alpha,\,\hat{M})} e_{M,j}{K_{\alpha,i}} \quad &\text{$i=j$},\\
v^{-(\alpha,\,\hat{M})} e_{M,j}{K_{\alpha,i}} \quad &\text{$i=m-1+j$},\\
e_{M,j}{K_{\alpha,i}} & {\text{otherwise};}
\end{cases}\\
&e_{M,i}e_{N,i}=\sum\limits_{[L]}v^{\lr{M,\,N}}g_{MN}^Le_{L,i};\\
&e_{M,i+1}e_{N,i}=\sum\limits_{[X],[Y]}v^{\lr{\hat{M}-\hat{X},\,\hat{X}-\hat{Y}}}\gamma_{MN}^{XY}K_{\hat{M}-\hat{X},i}e_{Y,i}e_{X,i+1};\\
&e_{M,i}e_{N,j}=e_{N,j}e_{M,i},~~i-j\neq0, 1~{\rm or}~m-1.
\end{flalign}
\end{proposition}

\begin{corollary}
Let $m=0$ or $m>2$.
Then for each $i\in\mathbb{Z}_m$,

$(1)$~there exists an embedding of algebras $\kappa_i:HD(\A)\hookrightarrow\mathcal {D}\mathcal {H}_m(\A)$ defined on generators by
$$K_{\alpha}^+\mapsto K_{\alpha,i+1},\quad K_{\alpha}^-\mapsto K_{\alpha,i},\quad \mu_M^+\mapsto e_{M,i+1},\quad \mu_M^-\mapsto e_{M,i};$$

$(2)$~there exists an embedding of algebras $\check{\kappa}_i:\check{H}D(\A)\hookrightarrow\mathcal {D}\mathcal {H}_m(\A)$ defined on generators by
$$\K_{\alpha}^+\mapsto K_{\alpha,i},\quad\K_{\alpha}^-\mapsto K_{\alpha,i+1},\quad \nu_M^+\mapsto e_{M,i},\quad \nu_M^-\mapsto e_{M,i+1}.$$
\end{corollary}
\begin{proof}
By Proposition \ref{BR} the defining relations of $HD(\A)$ and $\check{H}D(\A)$ are preserved under $\kappa_i$ and $\check{\kappa}_i$, respectively, we obtain that $\kappa_i$ and $\check{\kappa}_i$ are homomorphisms of algebras.
According to \cite[Proposition 2.7]{ZHC2}, we conclude that they are injective.
\end{proof}

As a first application of Theorem \ref{zyjg}, we have the following
\begin{theorem}\label{third}
Let $m=0$ or $m>2$. Then for each $i\in\mathbb{Z}_m$, there exists an embedding of algebras $\psi_i:D(\A)\hookrightarrow \mathcal {D}\mathcal {H}_m(\A)\otimes\mathcal {D}\mathcal {H}_m(\A)$ defined on generators by
\begin{flalign*}
\sK_\alpha^+\mapsto K_{\alpha,i+1}\otimes K_{\alpha,i},\quad \omega_M^+\mapsto\sum\limits_{[M_1],[M_2]}v^{\lr{M_1,M_2}}
\frac{a_{M_1}a_{M_2}}{a_M}g_{M_1M_2}^Me_{M_1,i+1}K_{\hat{M}_2,i+1}\otimes e_{M_2,i},
\end{flalign*} and
\begin{flalign*}
\sK_\alpha^-\mapsto K_{\alpha,i}\otimes K_{\alpha,i+1},\quad \omega_M^-\mapsto\sum\limits_{[M_1],[M_2]}v^{\lr{M_2,M_1}}
\frac{a_{M_1}a_{M_2}}{a_M}g_{M_2M_1}^Me_{M_1,i}\otimes e_{M_2,i+1}K_{\hat{M}_1,i+1}.
\end{flalign*}
\end{theorem}
\begin{proof}
For each $i\in\mathbb{Z}_m$, by the following commutative diagram
$$\xymatrix{D(\A)\ar@{^{(}->}[r]^-I\ar[rd]^-{\psi_i}&HD(\A)\otimes\check{H}D(\A)\ar@{^{(}->}[d]^-{{\kappa}_i\otimes\check{\kappa}_{i}}\\
&\mathcal {D}\mathcal {H}_m(\A)\otimes\mathcal {D}\mathcal {H}_m(\A)}$$
we complete the proof.
\end{proof}

\begin{remark}
As mentioned in Introduction, there is an isomorphism $\rho:D(\A)\to\mathcal {D}\mathcal {H}_2(\A)$, which is defined on generators by
\begin{flalign*}
\omega_M^+\mapsto \frac{E_M}{a_M}, \omega_M^-\mapsto \frac{F_M}{a_M},
\sK_\alpha^+\mapsto K_\alpha, \sK_\alpha^-\mapsto K_\alpha^*,
\end{flalign*}
where the notations $E_M$, $F_M$, $K_\alpha$ and $K_\alpha^*$ are the same as those in \cite{ZHC}. Hence, Theorem \ref{third} establishes a relation between the Bridgeland Hall algebra of $2$-cyclic complexes and that of $m$-cyclic complexes.

%
\end{remark}

\subsection{Derived Hall algebras}
The derived Hall algebra $\mathcal {D}\mathcal {H}(\A)$ of the bounded derived category of $\A$ was introduced in \cite{Toen} (see also \cite{XiaoXu}).
\begin{proposition}\label{DH} {\rm(\cite{Toen})}
$\mathcal {D}\mathcal {H}(\A)$ is an associative and unital $\mathbb{C}$-algebra generated by the elements in $\{Z_{M}^{[i]}~|~[M]\in\Iso(\A),~i\in \mathbb{Z}\}$ and the following relations:
\begin{flalign}
&Z_M^{[i]}Z_N^{[i]}=\sum\limits_{[L]}g_{MN}^LZ_L^{[i]};\\
&Z_M^{[i+1]}Z_N^{[i]}=\sum\limits_{[X],[Y]}q^{-\lr{{Y},{X}}}\gamma_{MN}^{XY}Z_Y^{[i]}Z_X^{[i+1]};\\
&Z_M^{[i]}Z_N^{[j]}=q^{(-1)^{i-j}\lr{N,M}}Z_N^{[j]}Z_M^{[i]}, \quad i-j>1.
\end{flalign}
\end{proposition}
According to \cite{SX}, we twist the multiplication in $\mathcal {D}\mathcal {H}(\A)$ as follows:
\begin{equation}Z_M^{[i]}\ast Z_N^{[j]}=v^{(-1)^{i-j}\lr{M,N}}Z_M^{[i]}Z_N^{[j]}.\end{equation}
The \emph{twisted derived Hall algebra} $\mathcal {D}\mathcal {H}_{\tw}(\A)$ is the same vector space as $\mathcal {D}\mathcal {H}(\A)$, but with the twisted multiplication. In order to relate the modified Ringel--Hall algebra, which is isomorphic to the corresponding Bridgeland Hall algebra if $\A$ has enough projectives, to derived Hall algebra, Lin \cite{LJ} introduced the \emph{completely extended twisted derived Hall algebra} $\mathcal {D}\mathcal {H}_{\tw}^{\ce}(\A)$.
\begin{definition}(\cite{LJ})\label{wanquan}
$\mathcal {D}\mathcal {H}_{\tw}^{\ce}(\A)$ is the associative and unital $\mathbb{C}$-algebra generated by the elements in $\{Z_M^{[i]}~|~[M]\in\Iso(\A),~i\in \mathbb{Z}\}$ and $\{K_\alpha^{[i]}~|~\alpha\in K(\A),~i\in \mathbb{Z}\}$, and the following relations:
\begin{flalign}
&K_\alpha^{[i]}K_\beta^{[i]}=K_{\alpha+\beta}^{[i]},~~K_\alpha^{[i]}Z_M^{[i]}=\begin{cases}
v^{(\alpha,\,\hat{M})} Z_M^{[i]}K_\alpha^{[i]}  &\text{$i=-1,0,$}\\
Z_M^{[i]}K_\alpha^{[i]}& {\text{otherwise};}
\end{cases}\\
&K_{\alpha}^{[i+1]}K_\beta^{[i]}=v^{(\alpha,\beta)}K_\beta^{[i]}K_{\alpha}^{[i+1]},~~K_\alpha^{[i]}K_\beta^{[j]}=K_\beta^{[j]}K_\alpha^{[i]},~~|i-j|>1;\\
&K_\alpha^{[i]}Z_M^{[i+1]}=\begin{cases}
v^{-(\alpha,\,\hat{M})} Z_M^{[i+1]}K_\alpha^{[i]} &\text{$i=-1,0,$}\\
Z_M^{[i+1]}K_\alpha^{[i]}& {\text{otherwise};}\\
\end{cases}\\
&K_\alpha^{[i]}Z_M^{[i-1]}=\begin{cases}
v^{-(\alpha,\,\hat{M})} Z_M^{[i-1]}K_\alpha^{[i]} &\text{$i=-1,0,$}\\
Z_M^{[i-1]}K_\alpha^{[i]}& {\text{otherwise};}\\
\end{cases}\\
&\text{For~~any~~}|i-j|>1,~~K_\alpha^{[i]}Z_M^{[j]}=\begin{cases}
v^{(-1)^j(\alpha,\,\hat{M})} Z_M^{[j]}K_\alpha^{[i]} &\text{$i=0$~~\text{and}~~$|j|>1$,}\\
v^{(-1)^{j+1}(\alpha,\,\hat{M})} Z_M^{[j]}K_\alpha^{[i]} &\text{$i=-1$~~\text{and}~~$|j+1|>1$,}\\
Z_M^{[j]}K_\alpha^{[i]} & {\text{otherwise};}\\
\end{cases}\\
&Z_M^{[i]}Z_N^{[i]}=\sum\limits_{[L]}v^{\lr{M,N}}g_{MN}^LZ_L^{[i]}\label{x2};\\
&Z_M^{[i+1]}Z_N^{[i]}=\sum\limits_{[X],[Y]}v^{-\lr{M,N}-\lr{Y,X}}\gamma_{MN}^{XY}Z_Y^{[i]}Z_X^{[i+1]};\\
&Z_M^{[i]}Z_N^{[j]}=v^{(-1)^{i-j}(M,N)}Z_N^{[j]}Z_M^{[i]}, \quad i-j>1.
\end{flalign}
\end{definition}

\begin{remark}
In Definition \ref{wanquan}, we have employed the linear Euler form, not the multiplicative Euler form used in \cite{LJ};
$K_\alpha^{[i]}$ and $Z_M^{[i]}$ here are equal to $K_\alpha^{[-i]}$ and $Z_M^{[-i]}$ in \cite{LJ}, respectively.
\end{remark}
Now we reformulate \cite[Theorem 5.3,Corollary 5.5]{LJ} as follows:
\begin{theorem}\label{tonggou}
Assume that $\A$ has enough projectives. Then
there exists an isomorphism of algebras $\phi:\mathcal {D}\mathcal {H}_{\tw}^{\ce}(\A)\to\mathcal {D}\mathcal {H}_0(\A)$ defined on generators (with $n>0$) by
$$Z_M^{[0]}\mapsto e_{M,0},\quad K_{\alpha}^{[n]}\mapsto K_{\alpha,n},$$
$$Z_M^{[n]}\mapsto v^{n\lr{M,M}}e_{M,n}\prod\limits_{i=1}^nK_{(-1)^i\hat{M},n-i},~~\text{and}~~Z_M^{[-n]}\mapsto v^{-n\lr{M,M}}e_{M,-n}\prod\limits_{i=0}^{n-1}K_{(-1)^{i+1}\hat{M},i-n}.$$
\end{theorem}

\begin{remark}
$(1)$~The inverse of $\phi$ in Theorem \ref{tonggou} is the homomorphism $\phi^{-1}:\mathcal {D}\mathcal {H}_0(\A)\to\mathcal {D}\mathcal {H}_{\tw}^{\ce}(\A)$ defined on generators (with $n>0$) by
$$e_{M,0}\mapsto Z_M^{[0]},\quad K_{\alpha,n}\mapsto K_{\alpha}^{[n]},$$
$$e_{M,n}\mapsto v^{-n\lr{M,M}}Z_M^{[n]}\prod\limits_{i=0}^{n-1}K_{(-1)^{n-i-1}\hat{M},i},~~\text{and}~~e_{M,-n}\mapsto v^{n\lr{M,M}}Z_M^{[-n]}\prod\limits_{i=1}^{n}K_{(-1)^{n-i}\hat{M},-i}.$$

$(2)$~Theorem \ref{tonggou} establishes the relation between the Bridgeland Hall algebra of bounded complexes over projectives of $\A$ and the derived Hall algebra of the bounded derived category $D^b(\A)$. In other word, one can realize the derived Hall algebra via Bridgeland's construction.
\end{remark}

As a second application of Theorem \ref{zyjg}, we have the following
\begin{theorem}
For each $i\in\mathbb{Z}$, there exists an embedding of algebras
$\varphi_i:D(\A)\hookrightarrow \mathcal {D}\mathcal {H}_{\tw}^{\ce}(\A)\otimes\mathcal {D}\mathcal {H}_{\tw}^{\ce}(\A)$. Explicitly,

$(1)$~if $i=-1$, $\varphi_i$ is defined on generators by
\begin{flalign*}
\sK_\alpha^+\mapsto K_{\alpha}^{[0]}\otimes K_{\alpha}^{[-1]},\quad \omega_M^+\mapsto\sum\limits_{[M_1],[M_2]}v^{\lr{M,M_2}}
\frac{a_{M_1}a_{M_2}}{a_M}g_{M_1M_2}^MZ_{M_1}^{[0]}K_{\hat{M}_2}^{[0]}\otimes Z_{M_2}^{[-1]}K_{\hat{M}_2}^{[-1]},
\end{flalign*}
\begin{flalign*}
\sK_\alpha^-\mapsto K_\alpha^{[-1]}\otimes K_\alpha^{[0]},\quad \omega_M^-\mapsto\sum\limits_{[M_1],[M_2]}v^{\lr{M,M_1}}
\frac{a_{M_1}a_{M_2}}{a_M}g_{M_2M_1}^MZ_{M_1}^{[-1]}K_{\hat{M}_1}^{[-1]}\otimes Z_{M_2}^{[0]}K_{\hat{M}_1}^{[0]};
\end{flalign*}

$(2)$~if $i=0$, $\varphi_i$ is defined on generators by
\begin{flalign*}
\sK_\alpha^+\mapsto K_{\alpha}^{[1]}\otimes K_{\alpha}^{[0]},\quad \omega_M^+\mapsto\sum\limits_{[M_1],[M_2]}v^{-\lr{\hat{M},\hat{M}_1}}
\frac{a_{M_1}a_{M_2}}{a_M}g_{M_1M_2}^MZ_{M_1}^{[1]}K_{\hat{M}_2}^{[1]}K_{\hat{M}_1}^{[0]}\otimes Z_{M_2}^{[0]},
\end{flalign*}
\begin{flalign*}
\sK_\alpha^-\mapsto K_\alpha^{[0]}\otimes K_\alpha^{[1]},\quad \omega_M^-\mapsto\sum\limits_{[M_1],[M_2]}v^{-\lr{\hat{M},\hat{M}_2}}
\frac{a_{M_1}a_{M_2}}{a_M}g_{M_2M_1}^MZ_{M_1}^{[0]}\otimes Z_{M_2}^{[1]}K_{\hat{M}_1}^{[1]}K_{\hat{M}_2}^{[0]};
\end{flalign*}

$(3)$~if $i<-1$, $\varphi_i$ is defined on generators by
\begin{flalign*}
&\sK_\alpha^+\mapsto K_{\alpha}^{[i+1]}\otimes K_{\alpha}^{[i]},\quad \sK_\alpha^-\mapsto K_{\alpha}^{[i]}\otimes K_{\alpha}^{[i+1]},\\
&\omega_M^+\mapsto\sum\limits_{[M_1],[M_2]}v^{x}
\frac{a_{M_1}a_{M_2}}{a_M}g_{M_1M_2}^MZ_{M_1}^{[i+1]}\prod_{j=1}^{-(i+1)}K_{(-1)^{i+j+1}\hat{M}_1}^{[-j]}K_{\hat{M}_2}^{[i+1]}\otimes Z_{M_2}^{[i]}\prod_{j=1}^{-i}K_{(-1)^{i+j}\hat{M}_2}^{[-j]}\\
&\quad\quad\quad\quad\quad\quad x=\lr{\hat{M}_1,\hat{M}_2-\hat{M}_1}-i(\lr{M_1,M_1}+\lr{M_2,M_2}),\\
&\omega_M^-\mapsto\sum\limits_{[M_1],[M_2]}v^{y}
\frac{a_{M_1}a_{M_2}}{a_M}g_{M_2M_1}^MZ_{M_1}^{[i]}\prod_{j=1}^{-i}K_{(-1)^{i+j}\hat{M}_1}^{[-j]}\otimes Z_{M_2}^{[i+1]}\prod_{j=1}^{-(i+1)}K_{(-1)^{i+j+1}\hat{M}_2}^{[-j]}K_{\hat{M}_1}^{[i+1]}\\
&\quad\quad\quad\quad\quad\quad y=\lr{\hat{M}_2,\hat{M}_1-\hat{M}_2}-i(\lr{M_1,M_1}+\lr{M_2,M_2});
\end{flalign*}

$(4)$~if $i>0$, $\varphi_i$ is defined on generators by
\begin{flalign*}
&\sK_\alpha^+\mapsto K_{\alpha}^{[i+1]}\otimes K_{\alpha}^{[i]},\quad \sK_\alpha^-\mapsto K_{\alpha}^{[i]}\otimes K_{\alpha}^{[i+1]},\\
&\omega_M^+\mapsto\sum\limits_{[M_1],[M_2]}v^{x}
\frac{a_{M_1}a_{M_2}}{a_M}g_{M_1M_2}^MZ_{M_1}^{[i+1]}\prod_{j=0}^{i}K_{(-1)^{i-j}\hat{M}_1}^{[j]}K_{\hat{M}_2}^{[i+1]}\otimes Z_{M_2}^{[i]}\prod_{j=0}^{i-1}K_{(-1)^{i-j-1}\hat{M}_2}^{[j]}\\
&\quad\quad\quad\quad\quad\quad x=\lr{\hat{M}_1,\hat{M}_2-\hat{M}_1}-i(\lr{M_1,M_1}+\lr{M_2,M_2}),
\end{flalign*}
\begin{flalign*}
&\omega_M^-\mapsto\sum\limits_{[M_1],[M_2]}v^{y}
\frac{a_{M_1}a_{M_2}}{a_M}g_{M_2M_1}^MZ_{M_1}^{[i]}\prod_{j=0}^{i-1}K_{(-1)^{i-j-1}\hat{M}_1}^{[j]}\otimes Z_{M_2}^{[i+1]}\prod_{j=0}^{i}K_{(-1)^{i-j}\hat{M}_2}^{[j]}K_{\hat{M}_1}^{[i+1]}\\
&\quad\quad\quad\quad\quad\quad y=\lr{\hat{M}_2,\hat{M}_1-\hat{M}_2}-i(\lr{M_1,M_1}+\lr{M_2,M_2}).
\end{flalign*}
\end{theorem}
\begin{proof}
By the following commutative diagram
$$\xymatrix{D(\A)\ar@{^{(}->}[r]^-{\psi_i}\ar[rd]_-{\varphi_i}&\mathcal {D}\mathcal {H}_0(\A)\otimes
\mathcal {D}\mathcal{H}_0(\A)\ar[d]_{\cong}^-{\phi^{-1}\otimes\phi^{-1}}\\
&\mathcal {D}\mathcal {H}_{\tw}^{\ce}(\A)\otimes\mathcal {D}\mathcal {H}_{\tw}^{\ce}(\A)}$$
we complete the proof.
\end{proof}

%

\end{document}